\documentclass[11pt, amsfonts]{amsart}

\usepackage{amsmath,amssymb,amsthm}
\usepackage[all]{xy}

\usepackage{color}

\numberwithin{equation}{section}

\SelectTips{cm}{}

\newtheorem{theorem}{Theorem}[section]
\newtheorem{proposition}[theorem]{Proposition}
\newtheorem{lemma}[theorem]{Lemma}

\theoremstyle{definition}

\newtheorem{example}[theorem]{Example}

\newtheorem{problem}[theorem]{Problem}

\theoremstyle{remark}
\newtheorem{remark}[theorem]{Remark}

\newcommand{\Z}{\mathbb{Z}}
\newcommand{\Q}{\mathbb{Q}}
\newcommand{\R}{\mathbb{R}}

\newcommand{\Ph}{\mathrm{Ph}}
\renewcommand{\lim}{\underset{\longleftarrow}{\mathrm{lim}}}
\newcommand{\limone}{\underset{\longleftarrow}{\mathrm{lim}}^1}

\title[Relative phantom maps and rational homotopy]{Relative phantom maps and rational homotopy}

\author{Daisuke Kishimoto}
\address{Department of Mathematics, Kyoto University, Kyoto, 606-8502, Japan}
\email{kishi@math.kyoto-u.ac.jp}

\author{Takahiro Matsushita}
\address{Department of Mathematical Sciences, University of the Ryukyus Nishihara-cho, Okinawa 903-0213, Japan}
\email{mtst@sci.u-ryukyu.ac.jp}


\begin{document}

\baselineskip.525cm

\maketitle

\begin{abstract}
  We generalize some results of Gray and McGibbon-Roitberg  on relations between phantom maps and rational homotopy to relative phantom maps. Since the $\limone$ and the profinite completion techniques do not apply to relative phantom maps, we develop new techniques.
\end{abstract}


\section{Introduction}

Let $X$ be a CW-complex of finite type. A map $X\to Y$ is called a \emph{phantom map} if its restriction to each skeleton of $X$ is null-homotopic. The reader can find comprehensive results and techniques for phantom maps in the survey \cite{M}. Recently, motivated by the de Bruijn-Erd\H{o}s theorem on colorings of infinite graphs, Iriye and the authors \cite{IKM} introduced a relative version of phantom maps defined as follows: Let $X$ be a CW-complex of finite type and $\varphi\colon B\to Y$ be a map between spaces. A map $X\to Y$ is called a \emph{relative phantom map} from $X$ to $\varphi$ if its restriction to each skeleton of $X$ lifts to $B$ through $\varphi$, up to homotopy. By definition, usual phantom maps from $X$ to $Y$ are relative phantom maps from $X$ to $\varphi$, and if $B$ is a point, then relative phantom maps to $\varphi$ are usual phantom maps from $X$ to $Y$. Thus relative phantom maps are a natural generalization of phantom maps, and we call usual phantom maps {\it absolute phantom maps} to distinguish them.

Phantom maps are known to be deeply related with rational homotopy. Here we recall two important results on phantom maps and rational homotopy. At the very first stage of the study of phantom maps, Gray \cite{G} found the following relation between phantom maps and rational homotopy. A phantom map is called trivial if it is null-homotopic.

\begin{proposition}[Gray \cite{G}]
  \label{Gray}
  If there is a non-trivial phantom map $X\to Y$, then $H^n(X;\Q)\ne 0$ and $\pi_{n+1}(Y)\otimes\Q\ne 0$ for some integer $n\ge 1$.
\end{proposition}

The $\limone$ technique enables us to study phantom maps algebraically, and by applying the $\limone$ technique, McGibbon and Roitberg \cite{MR} proved the following. To state their theorem, we set terms and notation. A space $X$ is called a \emph{finite type source} if it is a connected CW-complex of finite type, and a space $Y$ is called a \emph{finite type target} if $Y$ is path-connected and $\pi_n(Y)$ is finitely generated for each $n$. Let $\Ph(X,Y)$ be the pointed homotopy set of phantom maps from $X$ to $Y$.

\begin{theorem}[McGibbon-Roitberg \cite{MR}]
  \label{McGibbon-Roitberg}
  Let $X$ be a finite type source and $Y$ and $Y'$ be finite type targets. If a map $f\colon Y\to Y'$ induces a surjection $f_*\colon\pi_*(Y)\otimes\Q\to\pi_*(Y')\otimes\Q$ for $*\ge 2$, then
  $$f_*\colon\Ph(X,Y)\to\Ph(X,Y')$$
  is surjective.
\end{theorem}

\begin{remark}
  Proposition \ref{Gray} and Theorem \ref{McGibbon-Roitberg} are recovered by the result of Roitberg and Touhey \cite{RT} on a description of $\Ph(X,Y)$ using profinite completion.
\end{remark}

Our purpose is to generalize Proposition \ref{Gray} and Theorem \ref{McGibbon-Roitberg} to relative phantom maps. To state generalizations of the above results, we need the definition of the triviality of relative phantom maps. A relative phantom map $f \colon X \to Y$ from a CW-complex $X$ to a map $\varphi \colon B \to Y$ is called {\it trivial} if $f$ itself lifts to $B$, up to homotopy. When we regard an absolute phantom map $f \colon X \to Y$ as a relative phantom map to $* \hookrightarrow Y$, the triviality of relative phantom maps coincides with the triviality of absolute phantom maps.

In \cite{IKM}, the triviality of relative phantom maps is studied, and a criterion for the triviality of relative phantom maps to Postnikov sections is given in terms of rational homotopy, which is a partial generalization of Proposition \ref{Gray}. Thus it is possible that relative phantom maps are related with rational homotopy as well as absolute phantom maps. The aim of this paper is to verify such a relation by generalizing Proposition \ref{Gray} and Theorem \ref{McGibbon-Roitberg} to relative phantom maps. The key of the proofs of Proposition \ref{Gray} and Theorem \ref{McGibbon-Roitberg} is the congruence
$$\Ph(X,Y)\cong\limone[\Sigma X^n,Y]$$
and the key of the proof of the result of Roitberg and Touhey \cite{RT} which recovers Proposition \ref{Gray} and Theorem \ref{McGibbon-Roitberg} is the profinite completion technology. However, these do not apply to relative phantom maps, and so we must develop fairly new techniques.


Now we state the generalization of Proposition \ref{Gray} to relative phantom maps, which recovers the result of \cite{IKM} on Postnikov sections. Here we write $\pi_n(Y, B)$ to mean $\pi_n(M_\varphi, B)$, where $M_\varphi$ is the mapping cylinder of $\varphi \colon B \to Y$.

\begin{proposition}
  \label{main 1}
  Given a map $\varphi\colon B\to Y$, suppose that $\pi_1(B)$ acts trivially on $\pi_*(Y,B)$. If there is a non-trivial relative phantom map from $X$ to $\varphi$, then $H^n(X;\Q)\ne 0$ and $\pi_{n+1}(Y,B)\otimes\Q\ne 0$ for some integer $n\ge 1$.
\end{proposition}

To generalize Theorem \ref{McGibbon-Roitberg} to relative phantom maps to $\varphi\colon B\to Y$, there are two maps to be considered:
\begin{enumerate}
  \item the map $\Ph(X,\varphi)\to\Ph(X,f\circ\varphi)$ induced from $f\colon Y\to Y'$;
  \item the map $\Ph(X,\varphi\circ g)\to\Ph(X,\varphi)$ induced from $g\colon B'\to B$.
\end{enumerate}
However, the map (2) is not surjective in general when $g$ is a rational homotopy equivalence. Here we show such an example. In \cite{IKM}, a non-trivial relative phantom map $f\colon X(n)\to\R P^\infty$ to the inclusion $\R P^n\to\R P^\infty$ is constructed for $n\ge 3$, where $f$ is an isomorphism in $\pi_1$. Let $g\colon S^n\to\R P^n$ be the projection. For $n$ odd, $g$ is a rational homotopy equivalence, but the restriction of $f$ to each skeleton does not lift to $S^n$ through $\varphi\circ g$, up to homotopy, since $f$ is an isomorphism in $\pi_1$. Thus it is impossible to generalize Theorem \ref{McGibbon-Roitberg} with respect to the maps of type (2), so we will concentrate on the maps of type (1).

\begin{theorem}
  \label{main 2}
  Let $X$ be a finite type source, $B,Y$ and $Y'$ be simply-connected finite type targets, and $\varphi\colon B\to Y$ be a map. If $f\colon Y\to Y'$ is a rational homotopy equivalence, then
  $$f_*\colon\Ph(X,\varphi)\to\Ph(X,f\circ\varphi)$$
  is surjective.
\end{theorem}

By assuming $X$ to be a suspension, we can strengthen Theorem \ref{main 2}.

\begin{theorem}
  \label{main 3}
  Let $X$ be a finite type source, $B,Y$ and $Y'$ be simply-connected finite type targets, and $\varphi\colon B\to Y$ be a map. If $X$ is a suspension and $f\colon Y\to Y'$ induces a surjection $f_*\colon\pi_*(Y)\otimes\Q\to\pi_*(Y')\otimes\Q$, then
  $$f_*\colon\Ph(X,\varphi)\to\Ph(X,f\circ\varphi)$$
  is surjective.
\end{theorem}

The paper is organized as follows. In Section 2, we construct new relative phantom maps from old by applying Theorem \ref{main 2}. In Section 3, we describe the set of relative phantom maps $\Ph(X,\varphi)$ in terms of the limit, instead of $\limone$, of a certain tower of sets given by Moore-Postnikov tower of $\varphi$. In Section 4, we give proofs of the main results by using the description of $\Ph(X,\varphi)$ given in Section 3. In Section 5, we pose further problems on relative phantom maps and rational homotopy.

\textit{Acknowledgement:} The authors were supported in part by JSPS KAKENHI No. 17K05248 and No. 19K14536.


\section{Examples}


One can construct a relative phantom map from an non-trivial absolute phantom map, which is neither absolute nor trivial. Let $g\colon X\to Z$ be a non-trivial absolute phantom map and $B$ be a non-contractible space. Then the map $1\times g\colon B\times X\to B\times Z$ is a relative phantom map to the inclusion $B\to B\times Z$, which is neither absolute nor trivial. By applying Theorem \ref{main 2} to this construction, one gets new relative phantom maps to maps into H-spaces, which are neither absolute nor trivial.

\begin{proposition}
  \label{example 1}
  Let $Y$ be a simply-connected finite H-space of rank $\ge 2$. Then for some map $\varphi\colon B\to Y$, there is a relative phantom map to $\varphi$ which is neither absolute nor trivial.
\end{proposition}

\begin{proof}
  Since $Y$ is a finite H-space, there is a rational homotopy equivalence
  $$g\colon S^{2n_1-1}\times\cdots\times S^{2n_r-1}\to Y,$$
  where $n_1,\ldots,n_r\ge 2$ and $r\ge 2$. Since $Y$ is 0-universal in the sense of \cite{MT}, there is also a rational homotopy equivalence
  $$f\colon Y\to S^{2n_1-1}\times\cdots\times S^{2n_r-1}.$$
  Let $\varphi\colon S^{2n_1-1}\to Y$ be the restriction of $g$. {By composing rational self-equivalences of $S^{2n_1-1}$ and $S^{2n_1-1}\times\cdots\times S^{2n_r-1}$ with $\varphi$ and $f$ if necessary, we may assume that the composite
  $$S^{2n_1-1}\xrightarrow{\varphi}Y\xrightarrow{f}S^{2n_1-1}\times\cdots\times S^{2n_r-1}\to S^{2n_i-1}$$
  is non-trivial for $i=1$ and trivial for $i=2$. Then $f\circ\varphi$ is non-trivial in homology and is thought of as a map into $S^{2n_1-1}\times S^{2n_3-1}\times\cdots\times S^{2n_r-1}$.} As in \cite{M}, there is a non-trivial absolute phantom map $h\colon Z\to S^{2n_2-1}$ since $n_2\ge 2$. {Then one can define a map $\bar{h}=(f\circ\varphi)\times h\colon S^{2n_1-1}\times Z\to S^{2n_1-1}\times\cdots\times S^{2n_r-1}$ which is a relative phantom map to $f\circ\varphi$. Consider a commutative diagram
  $$\xymatrix{&S^{2n_1-1}\ar[r]\ar[d]^{f\circ\varphi}&\ast\ar[d]\\
  Z\ar[r]^(.25){\bar{h}\vert_Z}&S^{2n_1-1}\times\cdots\times S^{2n_r-1}\ar[r]&S^{2n_2-1}.}$$
  If $\bar{h}$ is a trivial relative phantom map, then $\bar{h}\vert_Z$ is also a trivial relative phantom map. So the composite of the bottom maps is a trivial absolute phantom map. But the composite of the bottom maps is $h$ which is a non-trivial absolute phantom map, a contradiction. Thus $\bar{h}$ is a non-trivial relative phantom map. Since $f\circ \varphi$ is non-trivial in homology, $\bar{h}$ is not absolute.}

  Since $f$ is a rational homotopy equivalence, it follows from Theorem \ref{main 2} that the induced map $f_*\colon\Ph(X,\varphi)\to\Ph(X,f\circ\varphi)$ is surjective. In particular, there is a relative phantom map $\tilde{h}\colon S^{2n_1-1}\times Z\to Y$ to $\varphi$ such that {$f\circ\tilde{h}\simeq\bar{h}$. Since $\bar{h}$ is neither trivial nor absolute, so is $\tilde{h}$.} Thus the proof is done.
\end{proof}

\begin{remark}
  As in \cite{MT}, one can moderate the condition on $Y$ in Proposition \ref{example 1} to that $Y$ is a simply-connected space having the homotopy type of a finite complex and a rational homotopy equivalence $S^{2n_1-1}\times\cdots\times S^{2n_r-1}\to Y$ for $r\ge 2$.
\end{remark}

\begin{example}
  {Let $G$ be a compact, simply-connected Lie group of rank $\ge 2$ and $\varphi\colon S^n\to G$ be any map of infinite order in $\pi_n$. There is always such a map $\varphi$ for $n=3$. Note that $\varphi$ extends to a rational homotopy equivalence $S^n\times P\to G$, where $P$ is a product of spheres. Then by Proposition \ref{example 1} and its proof, there is a relative phantom map to some non-zero multiple of $\varphi$ which is neither trivial nor absolute.}
\end{example}

Let $g\colon X\to Z$ be a non-trivial phantom map. If $B$ is not contractible, then the map $1\vee g\colon B\vee X\to B\vee Z$ is a relative phantom map to the inclusion $B\to B\vee Z$, which is neither absolute nor trivial. Recall that a simply-connected co-H-space has the rational homotopy type of a wedge of spheres. Moreover, they are 0-universal as in \cite{MT}. Then one can get the co-H-space analogue of Proposition \ref{example 1} by a quite similar proof.

\begin{proposition}
  \label{example 2}
  Let $Y$ be a simply-connected co-H-space having the homotopy type of a finite complex such that $\dim\widetilde{H}_*(Y;\Q)\ge 2$. Then for some map $\varphi\colon B\to Y$, there is a relative phantom map to $\varphi$ which is neither absolute nor trivial.
\end{proposition}

\begin{remark}
  As well as Proposition \ref{example 1}, one can moderate the condition on $Y$ in Proposition \ref{example 2} to that $Y$ is a simply-connected space having the homotopy type of a finite complex and a rational homotopy equivalence $Y\to S^{2n_1-1}\vee\cdots\vee S^{2n_r-1}$ for $r\ge 2$. See \cite{MT}.
\end{remark}


\section{Moore-Postnikov tower}

Recall that a Moore-Postninkov tower of a map $\varphi\colon B\to Y$ between path-connected spaces is a homotopy commutative diagram
$$\xymatrix{&\vdots\ar[d]\\
&Z_3\ar[d]\ar[ddr]\\
&Z_2\ar[d]\ar[dr]\\
B\ar[r]\ar[ur]\ar[uur]&Z_1\ar[r]&Y}$$
satisfying the following conditions:
\begin{enumerate}
  \item the composite $B\to Z_n\to Y$ is homotopic to $\varphi$ for each $n$;
  \item $B\to Z_n$ induces an isomorphism on $\pi_i$ for $i<n$ and a surjection for $i=n$;
  \item $Z_n\to Y$ induces an isomorphism on $\pi_i$ for $i>n$ and an injection for $i=n$;
  \item $Z_{n+1}\to Z_n$ is a fibration with fiber $K(\pi_{n+1}(Y,B),n)$.
\end{enumerate}
For example, a Moore-Postnikov tower of a map $X\to *$ is a Postnikov tower of $X$, and a Moore-Postnikov tower of a map $*\to X$ is a tower of connective covers of $X$.

If $B$ and $Y$ are connected CW-complexes, then any map $\varphi\colon B\to Y$ has a Moore-Postnikov tower. Moore-Postnikov towers are natural with respect to the underlying maps if they exist. A Moore-Postnikov tower is called {\it principal} if each $Z_{n+1}\to Z_n$ is a principal fibration. The following is well known. See \cite{Ha}, for example.

\begin{lemma}
  Given a map $\varphi\colon B\to Y$ between connected CW-complexes, if $\pi_1(B)$ acts trivially on $\pi_*(Y,B)$, then there is a principal Moore-Postnikov tower of $\varphi$.
\end{lemma}

We describe the set of relative phantom maps $\Ph(X,\varphi)$ in terms of the limit, instead of $\limone$, of a tower of sets given by a Moore-Postnikov tower of $\varphi$.

\begin{lemma}
  \label{Moore-Postnikov}
  Let $X$ be a CW-complex and $\{B\to Z_n\to Y\}_{n\ge 1}$ be a Moore-Postnikov tower of a map $\varphi\colon B\to Y$. Then $f\colon X\to Y$ is a relative phantom map to $\varphi$ if and only if it lifts to $Z_n$ for all $n\ge 1$.
\end{lemma}

\begin{proof}
  Let $f^n\colon X^n\to B$ be a lift of $f\vert_{X^n}$. There is no obstruction to extend the composite $X^n\xrightarrow{f^n}B\to Z_n$ to a lift of $f$, and so $f$ lifts to $Z_n$ for all $n\ge 1$ whenever it is a relative phantom map to $\varphi$. Let $g^n\colon X\to Z_n$ be a lift of $f$. There is no obstruction to lift $g^n$ to a map $X^{n-1}\to B$. Then $f$ has a lift $X\to Z_n$ for each $n$ whenever it is a relative phantom map to $\varphi$.
\end{proof}

\begin{proposition}
  \label{Ph}
  Let $X$ be a CW-complex and $\{B\to Z_n\xrightarrow{\varphi_n}Y\}_{n\ge 1}$ be a Moore-Postnikov tower of a map $\varphi\colon B\to Y$. Then
  $$\Ph(X,\varphi)\cong\lim\,\varphi_*([X,Z_n]).$$
\end{proposition}

\begin{proof}
  This immediately follows from Lemma \ref{Moore-Postnikov}.
\end{proof}


\section{Proofs}

We first prove a simple but useful lemma.

\begin{lemma}
  \label{limit}
  Let $\cdots\to F_3\xrightarrow{f_3}F_2\xrightarrow{f_2}F_1$ be a tower of non-empty finite sets. Then $\lim\,F_n\ne\emptyset$.
\end{lemma}

\begin{proof}
  Since all $F_n$ are finite, for some $x_1\in F_1$ there is an infinite sequence $1<n_1<n_2<\cdots$ such that there is $y_{n_i}\in(f_{n_i}\circ\cdots\circ f_2)^{-1}(x_1)\subset F_{n_i}$ for each $i$. Fix such elements $y_{n_1},y_{n_2},\ldots$. Since all $F_n$ are finite, there is $x_{n_1}\in(f_{n_1}\circ\cdots\circ f_2)^{-1}(x_1)$ such that there are infinitely many $n_i>n_1$ satisfying $y_{n_i}\in(f_{n_i}\circ\cdots\circ f_{n_1+1})^{-1}(x_{n_1})$. Similarly, there is $x_{n_2}\in(f_{n_2}\circ\cdots\circ f_{n_1+1})^{-1}(x_{n_1})$ such that there are infinitely many $n_i>n_2$ satisfying $y_{n_i}\in(f_{n_i}\circ\cdots\circ f_{n_2+1})^{-1}(x_{n_2})$. Repeating this procedure, we get a sequence $x_{n_1},x_{n_2},\ldots$ such that $(f_{n_i+1}\circ\cdots f_{n_{i+1}})(x_{n_{i+1}})=x_{n_i}$ for each $i$. From this sequence one easily gets an element of $\lim\,F_n$.
\end{proof}

We recall the Milnor exact sequence of a tower of fibrations (cf. \cite[Chapter IX, Corollary 3.2]{BK}).

\begin{lemma}
  \label{Milnor seq}
  Let $\cdots\to Y_3\to Y_2\to Y_1$ be a tower of fibrations. Then there is an exact sequence of pointed sets
  $$*\to\limone[X,\Omega Y_n]\to[X,\lim Y_n]\to\lim[X,Y_n]\to*.$$
\end{lemma}

The following fact is well known, so we omit the proof.

\begin{lemma}
  \label{finite}
  If $X$ is a finite complex and $Z$ is rationally contractible, then $[X,Z]$ is a finite set.
\end{lemma}

\begin{proof}
  [Proof of Proposition \ref{main 1}]
  Assume that for each $n \ge 1$, $H^n(X;\Q)=0$ or $\pi_{n+1}(Y,B)\otimes\Q=0$. Let $f$ be a relative phantom map from $X$ to $\varphi$. Let $\{B\to Z_n\to Y\}_{n\ge 1}$ be a principal Moore-Postnikov tower of $\varphi\colon B\to Y$ and $L_n$ be the pointed homotopy set of maps $X\to Z_n$ which are lifts of $f$. Since $f$ is a relative phantom map, $L_n$ is non-empty for each $n$ by Lemma \ref{Moore-Postnikov}. By assumption, there is a homotopy fibration $Z_n\to Z_{n-1}\to K(\pi_{n+1}(Y,B),n+1)$, and so one gets the Puppe exact sequence
  $$[X,K(\pi_{n+1}(Y,B),n)]\to[X,Z_n]\to[X,Z_{n-1}]\to[X,K(\pi_{n+1}(Y,B),n+1)]$$
  such that the inverse image of any element of $[X,Z_{n-1}]$ in $[X,Z_n]$ is an orbit of the action of $[X,K(\pi_{n+1}(Y,B),n)]$ on $[X,Z_n]$. By assumption and Lemma \ref{finite}, $[X,K(\pi_{n+1}(Y,B),n)]\cong H^n(X;\pi_{n+1}(Y,B))$ is finite, and so the inverse image of any element of $[X,Z_{n-1}]$ in $[X,Z_n]$ is finite. In particular, $L_n$ is finite for each $n$. Then one gets a tower of non-empty finite sets $\cdots\to L_3\to L_2\to L_1$, and thus by Lemma \ref{limit}, $\lim\,L_n\ne\emptyset$, or equivalently, there is a lift $f_n\colon X\to Z_n$ of $f$ for each $n$ such that the composite $X\xrightarrow{f_n}Z_n\to Z_{n-1}$ is homotopic to $f_{n-1}$. Since $\lim\,Z_n\simeq B$, it follows from Lemma \ref{Milnor seq} that the restriction map $[X,B]\to\lim\,[X,Z_n]$ is surjective. Then it follows that there is a map $\tilde{f}\colon X\to B$ such that the composite with the map $B\to Z_n$ is homotopic to $f_n$ for all $n$. In particular, we have $\varphi\circ\tilde{f}\simeq f$, and hence $f$ is a trivial relative phantom map. Thus any relative phantom map from $X$ to $\varphi$ is trivial, completing the proof.
\end{proof}

Hereafter, let $X$ be a finite type source and $B,Y,Y'$ be finite type targets. We start the proof of Theorem \ref{main 2}. By Theorem \ref{McGibbon-Roitberg}, one has:

\begin{lemma}
  \label{rationally contractible}
  If $Y$ is rationally contractible, then $\Ph(X,Y)=*$.
\end{lemma}

\begin{lemma}
  \label{fiber inclusion}
  Let $Y\xrightarrow{f}Y'\xrightarrow{g}Z$ be a homotopy fibration such that $f$ is a rational equivalence. Then for a map $\varphi\colon B\to Y$, $f_*\colon\Ph(X,\varphi)\to\Ph(X,f\circ\varphi)$ is surjective.
\end{lemma}

\begin{proof}
  Let $h\colon X\to Y'$ be a relative phantom map to $f\circ\varphi \colon B \to Y'$. Then for each $n$, there is a lift $h^n\colon X^n\to B$ of $h\vert_{X^n}$, where $X^n$ denotes the $n$-skeleton of $X$. Since $(g\circ h)\vert_{X^n}\simeq g\circ f\circ\varphi\circ h^n\simeq *$ for each $n$, $g\circ h$ is a phantom map, and hence by Lemma \ref{rationally contractible}, $g\circ h$ is null-homotopic. Thus one gets a map $\tilde{h}\colon X\to Y$ such that $f\circ\tilde{h}\simeq h$. We construct a relative phantom map $\bar{h}$ from $X$ to $\varphi$ such that $f\circ\bar{h}\simeq f\circ\tilde{h}$.

  Let $F_n\subset[X^n,Y]$ be the pointed homotopy set of composites of $\varphi$ and lifts $X^n\to B$ of $h\vert_{X^n}$ through $f\circ\varphi$, that is,
  $$F_n = \big\{ [\varphi \circ h'] \; | \; \textrm{$h' \colon X^n \to B$ is a map satisfying $f \circ \varphi \circ h' = h|_{X^n}$} \big\} .$$
Then one gets a tower of non-empty sets $\cdots\to F_3\to F_2\to F_1$. Let $h^n\colon X^n\to B$ be a lift of $h\vert_{X^n}$ through $f\circ\varphi$. Then $f\circ \tilde{h}\vert_{X^n}\simeq h\vert_{X^n}\simeq f\circ\varphi\circ h^n$. Consider the Puppe exact sequence
  $$[X^n,\Omega Z]\to[X^n,Y]\xrightarrow{f_*}[X^n,Y']\xrightarrow{g_*}[X^n,Z].$$
  Then there is $a_n\in[X^n,\Omega Z]$ such that $\tilde{h}\vert_{X^n}\cdot a_n\simeq\varphi\circ h^n$. Since $f\colon Y\to Y'$ is a rational equivalence, $Z$ is rationally contractible. Then by Lemma \ref{finite}, $[X^n,\Omega Z]$ is finite, and so each $F_n$ is finite. Thus by Lemma \ref{limit}, $\lim\,F_n\ne\emptyset$, or equivalently, there is $k^n\colon X^n\to Y$ for each $n$ such that $k^{n+1}\vert_{X^n}\simeq k^n$ and $k^n$ is homotopic to the composite of $\varphi$ and a lift of $h\vert_{X^n}$ through $f\circ\varphi$. Let $G_n$ be the subset of $[X^n,\Omega Z]$ consisting of $a_n$ such that $\tilde{h}\vert_{X^n}\cdot a_n\simeq k^n$. Then there is a tower of non-empty finite sets $\cdots\to G_3\to G_2\to G_1$, and by Lemma \ref{limit}, there is $b_n\in[X^n,\Omega Z]$ for each $n$ such that $\tilde{h}\vert_{X^n}\cdot b_n\simeq k^n$ and $b_{n+1}\vert_{X^n}\simeq b_n$. By Lemma \ref{Milnor seq}, the restriction map $[X,\Omega Z]\to\lim\,[X^n,\Omega Z]$ is surjective, and so there is a map $b\colon X\to\Omega Z$ such that $b\vert_{X^n}\simeq b_n$ for all $n$. Now we put $\bar{h}=\tilde{h}\cdot b$. Then $f\circ\bar{h}\simeq f\tilde{h}\simeq h$ and $\bar{h}\vert_{X^n}\simeq\tilde{h}\vert_{X^n}\cdot b_n\simeq k^n$, where $k^n$ lifts to $B$ through $\varphi$. Namely, $\bar{h} \colon X \to Y$ is a relative phantom map to $\varphi$. Thus the proof is completed.
\end{proof}

\begin{proof}
  [Proof of Theorem \ref{main 2}]
  Let $\{Y\to W_n\to Y'\}_{n\ge 1}$ be a principal Moore-Postnikov tower of a rational equivalence $f\colon Y\to Y'$, and let $g\colon X\to Y'$ be a relative phantom map to $f\circ\varphi \colon B \to Y'$. Since there is a homotopy fibration $W_i\to W_{i-1}\to K(\pi_{i+1}(Y',Y),i+1)$ for each $i$ where $\pi_{i+1}(Y',Y)$ is rationally trivial, it follows from Lemma \ref{fiber inclusion} that one inductively gets a relative phantom map $g_n\colon X\to W_n$ to $B\to W_n$ for each $n$ such that the composite $X\xrightarrow{g_n}W_n\to W_{n-1}$ is homotopic to $g_{n-1}$, where $g_0=g$. Since $\lim\,W_n\simeq Y$, it follows from Lemma \ref{Milnor seq} that the restriction map $[X,Y]\to\lim\,[X,W_n]$ is surjective. Then one gets a map $\tilde{g}\colon X\to Y$ such that the composite $X\xrightarrow{\tilde{g}}Y\to W_n$ is homotopic to $g_n$ for each $n$. In particular, by setting $n=0$, $f\circ\tilde{g}\simeq g$. It remains to show that $\tilde{g}$ is a relative phantom map to $\varphi$. Consider the commutative diagram
  $$\xymatrix{[X^n,B]\ar@{=}[r]\ar[d]^{\varphi_*}&[X^n,B]\ar[d]\\
  [X^n,Y]\ar[r]^(.42){\cong}&[X^n,W_{n+1}] ,}$$
  where the bottom map is a bijection. Since the composite $X^n\xrightarrow{\tilde{g}\vert_{X^n}}Y\to W_{n+1}$ is homotopic to $g_{n+1}\vert_{X^n}$ and $g_{n+1}$ is a relative phantom map to $B\to W_{n+1}$, this composite lifts to $B$. Thus by the commutative square above, one can see that $\tilde{g}\vert_{X^n}$ itself lifts to $B$ through $\varphi$. Since $n$ is arbitrary, $\tilde{g}$ is a relative phantom map to $\varphi$, completing the proof.
\end{proof}

Now we prove Theorem \ref{main 3}. We start with the following simple observation.

\begin{proposition}
  \label{adjoint}
  There is a natural isomorphism
  $$\Ph(\Sigma X,\varphi)\cong\Ph(X,\Omega\varphi).$$
\end{proposition}

\begin{proof}
  This immediately follows from the fact that $(\Sigma X)^{n+1}=\Sigma(X^n)$.
\end{proof}

We will use the following property of the 6-term exact sequence involving $\limone$. Let $1\to\{K_n\}_{n\ge 1}\to\{G_n\}_{n\ge 1}\to\{H_n\}_{n\ge 1}\to 1$ be an exact sequence of towers of groups. Then as in \cite[Chapter IX, Proposition 2.3]{BK}, there is an exact sequence of pointed sets
\begin{equation}
  \label{6-term}
  1\to\lim\,K_n\to\lim\,G_n\to\lim\,H_n\to\limone K_n\to\limone G_n\to\limone H_n\to *.
\end{equation}
Moreover, the group $\lim\,H_n$ acts on $\limone K_n$ such that if two elements of $\limone K_n$ are mapped to the same element of $\limone G_n$, then they are in the same orbit of the action of $\lim\,H_n$.

\begin{lemma}
  \label{action 6-term}
  Suppose that $1\to K_n\to G_n\to H_n\to 1$ is a central extension for each $n$, that is, the sequence is exact such that $K_n$ is mapped into the center of $G_n$. Then the group $\limone K_n$ acts on $\limone G_n$ so that for each pair of elements $x$ and $y$ in $\lim^1 G_n$, $g(x)=g(y)$ implies that there is $a\in\limone K_n$ satisfying $x\cdot a=y$. Here $g \colon \limone G_n \to \limone H_n$ denotes the map in \eqref{6-term}.
\end{lemma}

\begin{proof}
  Let $\cdots\xrightarrow{f_3}A_3\xrightarrow{f_2}A_2\xrightarrow{f_1}A_1$ be a tower of groups. Then $\prod_{n\ge 1}A_n$ acts on itself by
  $$(a_1,a_2,\ldots)\cdot(b_1,b_2,\ldots)=(b_1a_1f_2(b_2)^{-1},b_2a_2f_3(b_3)^{-1},\ldots),$$
  and by definition, $\limone A_n$ is the orbit set of this action. Then since $K_n$ is an abelian group for each $n$, $\limone K_n$ is an abelian group, and since $\prod_{n\ge 1}K_n$ is in the center of $\prod_{n\ge 1}G_n$, the coordinatewise action of $\prod_{n\ge 1}K_n$ on $\prod_{n\ge 1}G_n$ induces the action of $\limone K_n$ on $\limone G_n$. If $u,v\in\prod_{n\ge 1}G_n$ are mapped to the same element of $\prod_{n\ge 1}H_n$, then there is $w\in\prod_{n\ge 1}K_n$ such that $u\cdot w=v$. This descends to the desired property of the action of $\limone K_n$ on $\limone G_n$.
\end{proof}

Fix a map $\varphi\colon B\to Y$. Let $\{B\to Z_n\xrightarrow{\varphi_n}Y\}_{n\ge 1}$ be a Moore-Postnikov tower of $\varphi$ and $F$ be the homotopy fiber of $\varphi$.

\begin{proposition}
  \label{exact seq}
  There is an exact sequence of pointed sets
  $$[X,\Omega B]\xrightarrow{(\Omega\varphi)_*}\Ph(X,\Omega\varphi)\to\Ph(X,F)\to\limone[X,\Omega Z_n]$$
  which is natural with respect to $\varphi$.
\end{proposition}

\begin{proof}
  Let $F_n$ be the homotopy fiber of $\varphi_n$. Consider the Puppe exact sequence
  \begin{equation}
    \label{Puppe}
    [X,\Omega^2Y]\xrightarrow{(\Omega\delta_n)_*}[X,\Omega F_n]\to[X,\Omega Z_n]\xrightarrow{(\Omega\varphi_n)_*}[X,\Omega Y]
  \end{equation}
  where $\delta_n\colon\Omega Y\to F_n$ is the connecting map of a homotopy fibration $F_n\to Z_n\xrightarrow{\varphi_n}Y$. Let $H_n$ be the image of $[X,\Omega F_n]\to[X,\Omega Z_n]$. Then there is an exact sequence of towers of groups,
  $$1\to\{H_n\}_{n\ge 1}\to\{[X,\Omega Z_n]\}_{n\ge 1}\to\{(\Omega\varphi_n)_*([X,\Omega Z_n])\}_{n\ge 1}\to 1.$$
  Hence by \eqref{6-term}, there is an exact sequence of pointed sets
  $$\lim\,[X,\Omega Z_n]\to\lim\,(\Omega\varphi)_*([X,\Omega Z_n])\to\limone H_n\to\limone[X,\Omega Z_n].$$
  Since $\{\Omega B\to\Omega Z_n\xrightarrow{\Omega\varphi_n}Y\}_{n\ge 1}$ is a Moore-Postnikov tower of $\Omega\varphi$, it follows from Proposition \ref{Ph} that
  $$\Ph(X,\Omega\varphi)\cong\lim\,(\Omega\varphi)_*([X,\Omega Z_n]).$$
  Then since $\lim\,\Omega Z_n\simeq\Omega B$, the composite
  $$[X,\Omega B]\to\lim\,[X,\Omega Z_n]\to\lim\,(\Omega\varphi)_*([X,\Omega Z_n])$$
  is identified with the map $(\Omega\varphi)_*\colon[X,\Omega B]\to\Ph(X,\Omega\varphi)$, where $[X,\Omega B]\to\lim\,[X,\Omega Z_n]$ is surjective by Lemma \ref{Milnor seq}. Then it remains to identify $\limone H_n$ with $\Ph(X,F)$.

  By \eqref{Puppe} there is an exact sequence of towers of groups
  \begin{equation}
    \label{central}
    1\to\{(\Omega\delta_n)_*([X,\Omega^2Y])\}_{n\ge 1}\to\{[X,\Omega F_n]\}_{n\ge 1}\to\{H_n\}_{n\ge 1}\to 1,
  \end{equation}
  and hence by \eqref{6-term}, there is an exact sequence of pointed sets
  $$\limone(\Omega\delta_n)_*([X,\Omega^2Y])\to\limone[X,\Omega F_n]\to\limone H_n\to*.$$
  Since $\{(\Omega\delta_n)_*([X,\Omega^2Y])\}_{n\ge 1}$ is a tower of surjections, $\limone(\Omega\delta_n)_*([X,\Omega^2Y])=*$. By \cite{BG}, \eqref{central} is a central extension, and so by Lemma \ref{action 6-term}, $\limone H_n\cong\limone[X,\Omega F_n]$ as desired. The naturality of the exact sequence follows from the naturality of Moore-Postnikov towers.
\end{proof}

Now we prove Theorem \ref{main 3}. We start with the following algebraic lemma.


\begin{lemma}
  \label{4-lemma}
  Given a commutative diagram of pointed sets
  $$\xymatrix{A_1\ar[r]^{f_1}\ar[d]^\alpha&B_1\ar[r]^{g_1}\ar[d]^\beta&C_1\ar[r]^{h_1}\ar[d]^\gamma&D_1\ar[d]^\delta\\
  A_2\ar[r]^{f_2}&B_2\ar[r]^{g_2}&C_2\ar[r]^{h_2}&D_2}$$
  with exact rows, suppose that
  \begin{enumerate}
    \item $A_i$ and $B_i$ are groups and $f_i,\alpha,\beta$ are group homomorphisms for $i=1,2$;
    \item $A_i$ acts on $B_i$ for $i=1,2$ so that
    \begin{enumerate}
      \item if $g_i(b)=g_i(b')$ for $b,b'\in B_i$, then $b'=b\cdot a$ for some $a\in A_i$, and
      \item $\beta(b\cdot a)=\beta(b)\cdot\alpha(a)$ for $a\in A_1$ and $b\in B_1$.
    \end{enumerate}
  \end{enumerate}
  If $\alpha$ and $\gamma$ are surjective and $\delta^{-1}(*)=*$, then $\beta$ is surjective.
\end{lemma}

\begin{proof}
  Take any $b\in B_2$. Since $\gamma$ is surjective, there is $c\in C_1$ such that $\gamma(c)=g_2(b)$. Note that $\delta\circ h_1(c)=h_2\circ\gamma(c)=h_2\circ g_2(c)=*$. Then since $\delta^{-1}(*)=*$, we have $h_1(c)=*$, and so there is $b'\in B_1$ such that $g_1(b')=c$. Since $g_2\circ\beta(b')=\gamma\circ g_1(b')=\gamma(c)=g_2(b)$, there is $a\in A_2$ such that $b=\beta(b')\cdot a$. Since $\alpha$ is surjective, there is $a'\in A_1$ such that $\alpha(a')=a$, implying $b=\beta(b')\cdot a=\beta(b')\cdot\alpha(a')=\beta(b'\cdot a')$. Thus the proof is done.
\end{proof}

\begin{proof}
  [Proof of Theorem \ref{main 3}]
  Let $\{B\to Z_n\to Y\}_{n\ge 1}$ and $\{B\to Z_n'\to Y'\}_{n\ge 1}$ be Moore-Postnikov towers of $\varphi$ and $f\circ\varphi$, respectively. Since $\lim\,Z_n\simeq\lim\,Z_n'\simeq B$ and there is a homotopy commutative diagram
  $$\xymatrix{B\ar[r]\ar@{=}[d]&Z_n\ar[r]\ar[d]&Y\ar[d]^f\\
  B\ar[r]&Z_n'\ar[r]&Y'}$$
  for each $n$, it follows from Lemma \ref{Milnor seq} that there is a commutative diagram
  $$\xymatrix{\ast\ar[r]&\limone[A,\Omega Z_n]\ar[r]\ar[d]^g&[A,B]\ar[r]\ar@{=}[d]&\lim\,[A,Z_n]\ar[d]\ar[r]&\ast\\
  \ast\ar[r]&\limone[A,\Omega Z_n']\ar[r]&[A,B]\ar[r]&\lim\,[X,Z_n']\ar[r]&\ast}$$
  where $g$ is induced from the natural maps $Z_n\to Z_n'$. Then in particular, the map $g\colon\limone\,[A,\Omega Z_n]\to\limone[A,\Omega Z_n']$ satisfies $g^{-1}(*)=*$. By Proposition \ref{exact seq}, there is a commutative diagram with exact rows
  $$\xymatrix{[A,\Omega B]\ar[r]\ar@{=}[d]&\Ph(A,\Omega\varphi)\ar[r]\ar[d]^{(\Omega f)_*}&\Ph(A,F)\ar[r]\ar[d]^{\tilde{f}_*}&\limone\,[A,\Omega Z_n]\ar[d]^g\\
  [A,\Omega B]\ar[r]&\Ph(A,\Omega(f\circ\varphi))\ar[r]&\Ph(A,F')\ar[r]&\limone\,[A,\Omega Z_n']}$$
  where $F$ and $F'$ are the homotopy fibers of $\varphi$ and $f\circ\varphi$ respectively and $\tilde{f}\colon F\to F'$ is the induced map from $f$. By assumption, $f_*\colon\pi_*(Y,B)\otimes\Q\to\pi_*(Y',B)\otimes\Q$ is surjective, and this map is identified with $\tilde{f}_*\colon\pi_*(F)\otimes\Q\to\pi_*(F')\otimes\Q$. By the construction together with the property of the 6-term exact sequence of $\lim$ and $\limone$ mentioned before Lemma \ref{action 6-term}, the commutative diagram above satisfy all the conditions of Lemma \ref{4-lemma}. Thus the map $f_*\colon\Ph(A,\Omega\varphi)\to\Ph(A,\Omega(f\circ\varphi))$ is surjective. Therefore by Proposition \ref{adjoint}, the proof is completed.
\end{proof}


\section{Further problems}

In \cite{MR}, there are two main theorems (and their duals): one is Theorem \ref{McGibbon-Roitberg} and the other is the following.

\begin{theorem}
  \label{MR trivial}
  Let $Y$ be a finite type target. The following are equivalent:
  \begin{enumerate}
    \item $\Ph(X,Y)=*$ for every finite source $X$;
    \item $\Ph(K(\Z,n),Y)=*$ for every $n$;
    \item there is a rational homotopy equivalence $\prod_\alpha K(\Z,n_\alpha)\to\Omega_0Y$.
  \end{enumerate}
\end{theorem}

Let $\varphi\colon B\to Y$ be a map. We pose:

\begin{problem}
  Find whether or not there is a condition on rational homotopy which is equivalent to that for every finite type source $X$, every phantom map from $X$ to $\varphi$ is trivial.
\end{problem}

This may be a very hard problem without any clue, and so we pose a weak version. If there is a homotopy fibration $B\xrightarrow{\varphi}Y\xrightarrow{\pi}W$, then by \eqref{exact seq Ph}, there is an exact sequence of pointed sets
$$[X,B]\xrightarrow{\varphi_*}\Ph(X,\varphi)\xrightarrow{\pi_*}\Ph(X,W).$$
Thus for every finite type source $X$, every relative phantom map from $X$ to $\varphi$ is trivial if and only if the map $\pi_*\colon\Ph(X,\varphi)\to\Ph(X,W)$ is trivial.

\begin{problem}
  Find whether or not there is a condition on rational homotopy which is equivalent to the triviality of $\pi_*$ for every finite type source $X$.
\end{problem}

We include the possibility of non-existence of conditions in the two problems above because not every property of absolute phantom maps is generalized to relative phantom maps. For example, if $\Ph(X,Y)\ne*$, then its cardinality is uncountable. On the other hand, as mentioned in Section 1, a non-trivial relative phantom map $X(n)\to\R P^\infty$ to the inclusion $i_n\colon\R P^n\to\R P^\infty$, which is an isomorphism in $\pi_1$, is constructed for $n\ge 3$ in \cite{IKM}. Since $\Ph(X(n),i_n)\subset[X(n),\R P^\infty]\cong\Z/2$ and the constant map is a relative phantom map, the cardinality of $\Ph(X(n),i_n)$ is two.

To pose our last problem, we discuss about the map $\Ph(X,\Omega\varphi)\to\Ph(X,F)$ in Proposition \ref{exact seq}. Let $S\xrightarrow{\alpha}T\xrightarrow{\beta}U$ be a homotopy fibration. Then a map $f\colon X\to T$ is a relative phantom map to $\alpha$ if and only if $\beta\circ f\colon X\to U$ is an absolute phantom map. Then the exact sequence of pointed sets $[X,S]\xrightarrow{\alpha_*}[X,T]\xrightarrow{\beta_*}[X,U]$ restricts to an exact sequence of pointed sets
\begin{equation}
  \label{exact seq Ph}
  [X,S]\xrightarrow{\alpha_*}\Ph(X,\alpha)\xrightarrow{\beta_*}\Ph(X,U).
\end{equation}
Let $\delta\colon\Omega Y\to F$ be the connecting map of a homotopy fibration $F\to Y\xrightarrow{\varphi}B$. Then there is an exact sequence of pointed sets
$$[X,\Omega B]\xrightarrow{\Omega\varphi}\Ph(X,\Omega Y)\xrightarrow{\delta_*}\Ph(X,F)$$
and so one may pose:

\begin{problem}
  \label{connecting map}
  Is the map $\Ph(X,\Omega\varphi)\to\Ph(X,F)$ in Proposition \ref{exact seq} the induced map $\delta_*$?
\end{problem}

There is an affirmative evidence to this problem.

\begin{proposition}
  \label{delta}
  If $\varphi\colon B\to Y$ is a loop map or $X$ is a suspension, then the map $\Ph(X,\Omega\varphi)\to\Ph(X,F)$ in Proposition \ref{exact seq} is the induced map $\delta_*$
\end{proposition}

\begin{proof}
  Let $j\colon *\to F$ be the inclusion of a basepoint, $\delta\colon\Omega Y\to F$ be the connecting map, and $F\langle n\rangle$ be the $n$-connective cover of $F$. Then $\{*\to F\langle n\rangle\to F\}_{n\ge 1}$ is a Moore-Postnikov tower of $j$. Since there is a homotopy commutative diagram
  $$\xymatrix{\Omega B\ar[r]^{\Omega\varphi}\ar[d]&\Omega Y\ar[d]^\delta\\
  \ast\ar[r]^j&F}$$
  and either $\varphi$ is a loop map or $X$ is a suspension, it follows from the naturality of Proposition \ref{exact seq} that there is a commutative diagram of groups with exact rows
  $$\xymatrix{[X,\Omega B]\ar[r]\ar[d]&\Ph(X,\Omega\varphi)\ar[d]^{\delta_*}\ar[r]&\Ph(X,F)\ar@{=}[d]\ar[r]&\limone[X,\Omega Z_n]\ar[d]\\
  [X,*]\ar[r]&\Ph(X,j)\ar[r]&\Ph(X,F)\ar[r]&\limone[X,\Omega F\langle n\rangle].}$$
  Since $\lim\,F\langle n\rangle\simeq*$, it follows from Lemma \ref{Milnor seq} that $\limone[X,\Omega F\langle n\rangle]=*$, implying that the map $\Ph(X,j)\to\Ph(X,F)$ is an isomorphism. By definitioin, $\Ph(X,j)=\Ph(X,F)$, and so the map $\Ph(X,\Omega\varphi)\to\Ph(X,F)$ is identified with the induced map $\delta_*$.
\end{proof}

\end{document}